\documentclass[a4paper, 11pt]{article} % Font size (can be 10pt, 11pt or 12pt) and paper size (remove a4paper for US letter paper)
\pdfoutput=1

\usepackage{graphicx} % Required for including pictures
\usepackage{hyperref}
\usepackage{amsmath}
\usepackage{amssymb}
\usepackage{amsthm}

\newcounter{savefootnote}% Save footnote counter
\usepackage[T1]{fontenc} % Required for accented characters
\newcommand{\R}{\mathbb{R}}
\newcommand{\C}{\mathbb{C}}
\newcommand{\p}{\mathcal{P}}

\newtheorem{thm}{Theorem}
\newtheorem{remark}{Remark}

\makeatletter
\renewcommand\@biblabel[1]{\textbf{#1.}} % Change the square brackets for each bibliography item from '[1]' to '1.'
\renewcommand{\@listI}{\itemsep=0pt} % Reduce the space between items in the itemize and enumerate environments and the bibliography

\renewcommand{\maketitle}{ % Customize the title - do not edit title and author name here, see the TITLE block below
	\begin{flushright} % Right align
		{\LARGE\@title} % Increase the font size of the title
		
		\vspace{50pt} % Some vertical space between the title and author name
		
		{\large\@author Dimitris Vartziotis 
			\footnote{TWT GmbH Science \& Innovation, Department for Mathematical Research, Ernsthaldenstra{\ss}e 17, 70565 Stuttgart, Germany}\setcounter{savefootnote}{\value{footnote}}\footnote{NIKI Ltd. Digital Engineering, Research Center, 205 Ethnikis Antistasis Street, 45500 Katsika, Ioannina, Greece}
			\footnote{Corresponding author. E-mail address: dimitris.vartziotis@nikitec.gr},
			Juri Merger\setcounter{footnote}{\value{savefootnote}}\footnotemark[\value{footnote}]} % Author name
		\\\@date % Date
		
		\vspace{40pt} % Some vertical space between the author block and abstract
	\end{flushright}
}

\title{\textbf{GETMe.anis: On geometric polygon transformations leading to anisotropy}\\ % Title
} % Subtitle

\author{\textsc{} % Author
	\\{\textit{}}} % Institution

\date{\today} % Date

\begin{document}

	\maketitle % Print the title section
	
%----------------------------------------------------------------------------------------
%	ABSTRACT AND KEYWORDS
%----------------------------------------------------------------------------------------

\renewcommand{\abstractname}{Summary} % Uncomment to change the name of the abstract to something else

\begin{abstract}
	This work contributes to the analysis of linear geometric polygon transformations with the aim to force a desired shape by an iterative procedure.
\end{abstract}

\hspace*{3,6mm}\textit{Keywords:} Linear algebra, polygon transformations, anisotropy % Keywords

\vspace{30pt} % Some vertical space between the abstract and first section

\section{Introduction}
For several decades the finite element method is state-of-the-art when it comes to numerical simulations in various engineering fields.
Hereby, the discretization of the continuous problem of solving partial differential equations takes places on two sides.
On the one hand, the week formulation considering an infinite dimensional space of trial functions is replaced by a finite dimensional space consisting of polygonal trial functions.
On the other hand, also the domain of the partial differential equation is discretized by a polygonal mesh.
In practice it turned out that the latter point is the most time consuming step in a numerical simulation.
The convergence and error properties of the numerical approximation of the continuous problem is strictly connected to the finite element mesh quality.
Therefore, a huge effort is being done in order to derive efficient mesh generation and smoothing methods.
One direction is the geometric element transformation methods that improved the mesh quality by an iterative application of a geometric transformation to the elements of the mesh \cite{VartziotisAthanasiadisGoudasWipper2008, VartziotisWipperGETMeMixed2009,  VartziotisPapadrakakis2013, VartziotisWipperPapadrakakis2013, VartziotisBohnet2014}. 
See also \cite{VartziotisWipper2009GenNap, VartziotisWipper2009MP, VartziotisWipper2010, VartziotisHuggenberger2012, VartziotisBohnet2013, VartziotisBohnet2014b, Douglas1940, Petr1908, Merriell1965, Neumann1942, Shephard2003, Schuster1998, DingHittZhang2003, Davis1994, Davis1979} for the analysis of various aspects of geometric polygon transformations.
In line of this research the current work is extending the topic to anisotropic limit polygons, as the aim of the previous work is to generate meshes	with as regular as possible elements.

\section{Linear polygon-transformation with variations along different edges}
\label{seq:lin_trans_diff_lambda}
For a polygon $\p \in \C^n$ and a vector $w \in \C^n$ we define the following transformation
$$ \p = (z_1, \ldots, z_n)^\top \mapsto M\, \p,$$
where the transition matrix $M$ is defined as follows:
\begin{equation}
 M = \begin{pmatrix} 1-w_1 & w_1 & 0 & \cdots \\ 0 & 1 - w_2 & w_2 & \vdots \\ \vdots & \ddots & \ddots & \vdots \\ w_n & \cdots & 0 & 1-w_n  \end{pmatrix}
 \label{eq:M}
\end{equation}

This transformation maps a vertex $z_i$ to a point that is constructed as follows.
Consider the line that is orthogonal to the line connecting $z_i$ and $z_{i+1}$, whose intersection point with that line has a distance of $\Re(w_i) \, |z_{i+1}-z_{i}|$ from $z_i$, where a negative distance value represents a point on opposite side of $z_{i+1}$.
Further, we construct a line through $z_i$ whose angle with the line between $z_i$ and $z_{i+1}$ is the same as the angle of $w_i$, which equals $\arctan \left( \frac{\Im (w_i)}{\Re (w_i)}\right) $. 
The point $z_i$ is mapped on the intersection point of these two lines, which happens to be $z_i + w_i(z_{i+1}-z_i)$.
This is a generalization of the $\lambda$-$\theta$ transformation for not only $\lambda \in (0,1)$ and $\theta \in [0,\frac{\pi}{2}]$, but also for $\lambda \in \R$ and $\theta \in [0,2 \pi]$.
	
The key for the investigation of the limit behavior of the sequence $M^n \p$ are the eigenvalues and -vectors of the matrix $M$. As the eigenvectors of $M$ and $M-I$ are the same and the eigenvalues of $M-I$ are given by $\mu_k - 1$, where $\mu_k$ are eigenvalues of $M$, we investigate the matrix $M-I$.
Therefore, we compute the characteristic polynom of $M-I$, which is given by
\begin{align*} p(x) =& \det\left(xI - (M-I)\right) = \det\begin{pmatrix} x+w_1 & -w_1 & 0 & \cdots \\ 0 & x + w_2 & -w_2 & \vdots \\ \vdots & \ddots & \ddots & \vdots \\ -w_n & \cdots & 0 & x+w_n  \end{pmatrix} \\
=& \prod\limits_{i = 1}^n \left(x+w_i\right) + (-1)^{n-1} \prod\limits_{i = 1}^n (-w_i) \\
=& \prod\limits_{i = 1}^n \left(x+w_i\right) -  \prod\limits_{i = 1}^n w_i 
\end{align*}	

There are two cases:
\begin{enumerate}
	\item 
When one of the complex number $w_i$ equals zero, we can easily compute the eigenvalues of $M-I$, as we have
$$ p(x) = x \,  \prod\limits_{i = 2}^n \left(x+w_i\right), $$
where we w.l.o.g. have assumed that $w_1 = 0$ holds. 
Hence, the eigenvalues of $M$ are given by $\mu_i = 1-w_i$.
The eigenvector corresponding to $\mu_k$ is given by $v_k = (0,1,\frac{w_2 - w_k}{w_2},\frac{w_2 - w_k}{w_2}\frac{w_3 - w_k}{w_3},\ldots, \prod\limits_{j=2}^{k-1} \frac{w_j - w_k}{w_j}, 0 , \ldots)^\top$.
This can be verified by
\begin{align*}
M\, v_k = \begin{pmatrix} w_1 \\ 1-w_2 + w_2 \frac{w_2 - w_k}{w_2} \\ (1-w_3)\frac{w_2 - w_k}{w_2} + w_3  \frac{w_2 - w_k}{w_2}\frac{w_3 - w_k}{w_3} \\ \vdots  \\ (1-w_i) \prod\limits_{j=2}^{i-1} \frac{w_j - w_k}{w_j} + w_i  \prod\limits_{j=2}^{i} \frac{w_j - w_k}{w_j} \\ \vdots \end{pmatrix}
= \begin{pmatrix} 0 \\ 1 - w_k\\ (1-w_k)\frac{w_2 - w_k}{w_2}\\ \vdots  \\ (1-w_k) \prod\limits_{j=2}^{i-1} \frac{w_j - w_k}{w_j}  \\ \vdots \end{pmatrix}
\end{align*}

Note also that from the equation $(M-I)v = -w_k v$, we can conclude
$$ -w_i v_i + w_i v_{i+1} = - w_k v_i \qquad \text{for } i = 1,\ldots,n,$$
which is equivalent to 
$$ v_{i+1} = \frac{w_i - w_k}{w_i} v_i \qquad \text{for } i = 1,\ldots,n$$
for all $i$ such that $w_i \neq 0$.

Note that the transformation described above is translation and scale invariant.
Hence, we can transform each desired polygon w.l.o.g. such that $v_1 = 0$ and it follows that $w_1 = 0$, which means that we can compute the eigenvalues exactly.
\item The case, where $w_i \neq 0$ for all $i = 1,\ldots,n$ is more involved and still an open question up to now.
\end{enumerate}

\section{Constructing weights for an arbitrary limit polygon} 
In this section we analyse the inverse problem. 
Given a desired limit polygon $v \in \C^n$ we want to construct weights $w_i$ ($i = 1, \ldots, n$), such that the sequence $z^{(k)} := M^k z^{(0)}$ converges to $v$ as $n$ tends to infinity, where $M$ is given by $\eqref{eq:M}$.

This is the case, when $v$ is an eigenvector of the matrix $M$ and the corresponding eigenvalue has an absolute value larger than all other eigenvalues.
Here, we can ignore the eigenvalue for the eigenvector $(1,\ldots,1)^\top$, as the corresponding $n$-gon is the pointed polygon, which has no effect on the shape.

Assume that $v \in \C$ is given and define the auxiliary weights $\widetilde{w}_i =  \frac{v_i}{v_{i+1}-v_i}$.
Then, $v$ is an eigenvector with eigenvalue $2$ of the corresponding transformation matrix $\widetilde{M}$.
This follows by
\begin{align*}
(\widetilde{M} v)_i &= (1-\widetilde{w}_i) v_i + \widetilde{w}_i v_{i+1} \\
          &= \left(1-\frac{v_i}{v_{i+1}-v_i}\right) v_i + \frac{v_i}{v_{i+1}-v_i} v_{i+1} \\
          &= v_i - \frac{v_i}{v_{i+1}-v_i} v_i + \frac{v_i}{v_{i+1}-v_i} v_{i+1} \\
          &= 2v_i.
\end{align*}
Let us denote by $1+\mu_i$ ($i = 1,\ldots,n-1$) the remaining eigenvalues of $\widetilde{M}$, where we have w.l.o.g. $\mu_1 = 1$, as we have $\widehat{M} (1,\ldots,1)^\top = (1,\ldots,1)^\top$.
It is in general not true that we have $|1+\mu_i| < 2 $ for all $i = 2,\ldots,n-1$.
Hence, we use transformed weights $w_i := \lambda \widetilde{w}_i$ for a complex number $\lambda \in \C$.
Observe that the corresponding matrix $M$ has the same eigenvectors as $\widetilde{M}$ and but the  eigenvalues are given by $\mu_1 = 1$, $1+\lambda \mu_2$, $\ldots$, $1+\lambda \mu_{n-1}$, and $1+\lambda$, where $\lambda$ is the eigenvalue to the eigenvector $v$, which is the desired $n$-gon.
Our aim is now to find a $\lambda$ such that we have
$$ |1+\lambda| > |1+ \lambda \mu_i| \qquad \text{for all } i = 2, \ldots , n-1 $$

Therefore, we have the following theorem.
\begin{thm}
	Let $\mu$ be a complex number. The set $\Lambda_\mu$ of all complex numbers $\lambda \in \C$ such that
	$$ |1+\lambda|^2 > |1 + \lambda \, \mu|^2 $$
	holds is given by 
	\begin{itemize}
		\item Case 1: ($|\mu| < 1$)  the exterior of the circle with radius $\omega$ and midpoint $\omega$, where we have $\omega=\frac{1-\overline{\mu}}{|\mu|^2-1} $ , i.e. $$\Lambda_\mu = \{\lambda \in \C :~ |\lambda - \omega| >  |\omega|\},$$
		\item Case 2: ($|\mu| = 1$)  an open half plane of all complex number, whose angle to the point $\overline{\mu} - 1$ is less than $90^\circ$, i.e.
		$$\Lambda_\mu = \left\{\lambda \in \C :~  \measuredangle \lambda \in \left(\measuredangle (\overline{\mu} - 1) - \frac{\pi}{2}, \measuredangle(\overline{\mu} - 1) + \frac{\pi}{2} \right) \right\},$$
		\item Case 3: ($|\mu| > 1$)  the interior of the circle with radius $\omega$ and midpoint $\omega$, where we have $\omega=\frac{1-\overline{\mu}}{|\mu|^2-1} $ , i.e. $$\Lambda_\mu = \{\lambda \in \C :~ |\lambda - \omega| <  |\omega|\}.$$
	\end{itemize}
	\label{th:th}
\end{thm}
\begin{proof}
	Algebraic manipulations.
\end{proof}

\begin{thm}
	Let $\mu_i$ be defined as above. Assume that the set $\bigcap_{i = 2}^{n-1}\Lambda_{\mu_i}$ is not empty. Then, there exist weights $w_i$ such that $v$ is the limit polygon of the transformation matrix $M$.
	\label{th:th2}
\end{thm}
\begin{proof}
	Choose $\lambda \in \bigcap_{i = 2}^{n-1}\Lambda_{\mu_i}$ and define the weights as above. 
	Then, we have $|1+\lambda| > |1+ \lambda \mu_i|$ for all $i = 2, \ldots , n-1$.
	This has the statement of the theorem as a consequence.
\end{proof}

\subsection{Triangles ($n=3$)}
For transformations on triangles ($n=3$) the intersection in Theorem \autoref{th:th2} is trivially non-empty as there is only one competing eigenvalue. Consequently, there exist weights such that any triangle is the limit of the polygon sequence $M^k z^{(0)}$ regardless the initial triangle $z^{(0)}$.
In particular, we can choose $\lambda = - \mu_2^{-1}$ and transform the only competing eigenvalue to zero.
This means that any initial triangle is transformed in exactly one step to the desired triangle.
Hence, we have the following algorithm:
\begin{enumerate}
	\item Choose arbitrary $v \in \C^3$.
	\item Compute temporary transformation weight $\widetilde{w}_i = \frac{v_i}{v_{i+1}-v_i}$.
	\item Compute the third eigenvalue of $M-I$, which is given by $\mu =  \sum\limits_{i=1}^3 \widetilde{w}_i \widetilde{w}_{i+1}$. The other two eigenvalues are $0$ and $1$.
	\item Compute scaled weights $w_i = - \widetilde{w}_i \, \mu^{-1}$.
\end{enumerate}
		
\subsection{Quadrangles ($n=4$)}
In this subsection we concentrate on quadrangles.
Here, we have two competing eigenvalues which we denote by $\mu_2$ and $\mu_3$.
We have to analyze the possibility that the intersection $\Lambda_{\mu_2} \cap \Lambda_{\mu_2}$ is empty. 
This can happen for example if we have case 3 in Theorem \autoref{th:th} for both eigenvalues and the line between $\omega_2 = \frac{1-\overline{\mu_2}}{|\mu_2|^2-1}$ and $\omega_3 = \frac{1-\overline{\mu_3}}{|\mu_3|^2-1}$ covers the origin $z = 0$, e.g. $\omega_2 = - \omega_3$.
In this situation the corresponding circles do not intersect.
The same can happen for half planes in case 2.

\begin{remark}
	It is up to now an open question whether there are quadrangles $v \in \C^4$ such that these cases occur and the intersection is empty. 
	In this case we have proven that there exist no weights $w_i$ such that $v$ is the limit of the polygon sequence.
\end{remark}

Assuming that the intersection is non-empty we can choose $\lambda$ in the following way:
\vspace{0.5cm}

\noindent
\begin{tabular}{|l|c|} \hline
	Case: & $\lambda$ \\ \hline
	$|\mu_2| < 1 ~ \wedge ~ |\mu_3| < 1$ & $ 2~(\omega_2 + \omega_3) $ \\ \hline
	$|\mu_2| < 1 ~ \wedge ~ |\mu_3| = 1$ & $ 3~|\omega_2|~(\overline{\mu}_2-1) $ \\ \hline
	$|\mu_2| < 1 ~ \wedge ~ |\mu_3| > 1$ & $ \omega_2 + \frac{\omega_3 - \omega_2}{| \omega_3 - \omega_2 |} \frac{|\omega_3 - \omega_2| + |\omega_3| + |\omega_2|}{2} $ \\ \hline
	$|\mu_2| = 1 ~ \wedge ~ |\mu_3| < 1$ & $ 3~|\omega_2|~(\overline{\mu}_3-1) $ \\ \hline
	$|\mu_2| = 1 ~ \wedge ~ |\mu_3| = 1$ & $ \overline{\mu}_2 -1 + \overline{\mu}_3 - 1 $ \\ \hline
	$|\mu_2| = 1 ~ \wedge ~ |\mu_3| > 1$ & $ \omega_2 + \frac{\overline{\mu}_2 - 1}{|\overline{\mu}_2 - 1|} \frac{x + x + |\omega_2|}{2} $\\ \hline
	$|\mu_2| > 1 ~ \wedge ~ |\mu_3| < 1$ & $ \omega_2 + \frac{\omega_3 - \omega_2}{| \omega_3 - \omega_2 |} \frac{|\omega_3 - \omega_2| + |\omega_3| + |\omega_2|}{2} $ \\ \hline
	$|\mu_2| > 1 ~ \wedge ~ |\mu_3| = 1$ & $ \omega_2 + \frac{\overline{\mu}_3 - 1}{|\overline{\mu}_3 - 1|} \frac{x + x + |\omega_2|}{2} $ \\ \hline
	$|\mu_2| > 1 ~ \wedge ~ |\mu_3| > 1$ & $ \omega_2 + \frac{(\omega_3 - \omega_2)}{|\omega_3 - \omega_2|} \frac{|\omega_2| + |\omega_3 - \omega_2| - |\omega_3|}{2}$ \\ \hline
\end{tabular}
\vspace{0.5cm}

\noindent where $x$ is given by the distance of $\omega_2$ and the line through the origin which is orthogonal to the vector $\overline{\mu}_2 - 1$.

\section{Discussion}
This work extends the idea of geometric polygon transformations to regularize finite element meshes beyond regular elements.
In use cases where anisotropy inside a mesh is desired, alternative transformations can be used to force finite elements to obtain a certain anisotropy.
As previously shown this is always possible for triangles and mostly also for quadrangles.

A further approach can be to temporarily change the metric for the GETMe algorithm such that it regularizes the elements of the mesh with respect to a metric that induces an anisotropy as desired.

\bibliographystyle{plain}
%\bibliography{literature}
\bibliography{Vartziotis_Merger_-_GETMe.anis_On_geometric_polygon_transformations_leading_to_anisotropy}

\begin{thebibliography}{10}

\bibitem{Davis1979}
Philip~J. Davis.
\newblock Cyclic transformations of $n$-gons and related quadratic forms.
\newblock {\em Linear Algebra and its Applications}, 25:57--75, 1979.

\bibitem{Davis1994}
Philip~J. Davis.
\newblock {\em Circulant Matrices}.
\newblock Chelsea Publishing, 2 edition, 1994.

\bibitem{DingHittZhang2003}
Jiu Ding, L.~Richard Hitt, and Xin-Min Zhang.
\newblock Markov chains and dynamic geometry of polygons.
\newblock {\em Linear Algebra and its Applications}, 367:255--270, 2003.

\bibitem{Douglas1940}
Jesse Douglas.
\newblock On linear polygon transformations.
\newblock {\em Bulletin of the American Mathematical Society}, 46:551--560,
  1940.

\bibitem{Merriell1965}
David Merriell.
\newblock Further remarks on concentric polygons.
\newblock {\em American Mathematical Monthly}, 72:960--965, 1965.

\bibitem{Neumann1942}
B.H. Neumann.
\newblock A remark on polygons.
\newblock {\em Journal of the London Mathematical Society}, s1-17:165--166,
  1942.

\bibitem{Petr1908}
K.~Petr.
\newblock {Ein Satz {\"u}ber Vielecke}.
\newblock {\em {Archiv der Mathematik und Physik: mit besonderer R{\"u}cksicht
  auf die Bed{\"u}rfnisse der Lehrer an h{\"o}heren Unterrichtsanstalten}},
  13:29--31, 1908.

\bibitem{Schuster1998}
Wolfgang Schuster.
\newblock {Regularisierung von Polygonen}.
\newblock {\em Mathematische Semesterberichte}, 45(1):77--94, 1998.

\bibitem{Shephard2003}
G.C. Shephard.
\newblock Sequences of smoothed polygons.
\newblock In Andr{\'a}s Bezdek, editor, {\em Discrete Geometry}, Pure and
  Applied Mathematics, pages 407--430. Marcel Dekker, 2003.

\bibitem{VartziotisAthanasiadisGoudasWipper2008}
Dimitris Vartziotis, Theodoros Athanasiadis, Iraklis Goudas, and Joachim
  Wipper.
\newblock {Mesh smoothing using the Geometric Element Transformation Method}.
\newblock {\em Comput. Methods Appl. Mech. Engrg.}, 197(45--48):3760--3767,
  2008.

\bibitem{VartziotisBohnet2013}
Dimitris Vartziotis and Doris Bohnet.
\newblock Regularizations of non-euclidean polygons.
\newblock {\em arXiv:1312.2500 [math.MG]}, 2013.

\bibitem{VartziotisBohnet2014}
Dimitris Vartziotis and Doris Bohnet.
\newblock Convergence properties of a geometric mesh smoothing algorithm.
\newblock {\em arXiv:1411.3869 [math.NA]}, 2014.

\bibitem{VartziotisBohnet2014b}
Dimitris Vartziotis and Doris Bohnet.
\newblock Existence of an attractor for a geometric tetrahedron transformation.
\newblock {\em Differential Geometry and its Applications}, 49:197 -- 207,
  2016.

\bibitem{VartziotisHuggenberger2012}
Dimitris Vartziotis and Simon Huggenberger.
\newblock Iterative geometric triangle transformations.
\newblock {\em Elem. Math.}, 67(2):68--83, 2012.

\bibitem{VartziotisPapadrakakis2013}
Dimitris Vartziotis and Manolis Papadrakakis.
\newblock Improved {GETM}e by adaptive mesh smoothing.
\newblock {\em Computer Assisted Methods in Engineering and Science},
  20:55--71, 2013.

\bibitem{VartziotisWipper2009MP}
Dimitris Vartziotis and Joachim Wipper.
\newblock Classification of symmetry generating polygon-transformations and
  geometric prime algorithms.
\newblock {\em Math. Pannon.}, 20(2):167--187, 2009.

\bibitem{VartziotisWipper2009GenNap}
Dimitris Vartziotis and Joachim Wipper.
\newblock {On the Construction of Regular Polygons and Generalized Napoleon
  Vertices}.
\newblock {\em Forum Geom.}, 9:213--223, 2009.

\bibitem{VartziotisWipperGETMeMixed2009}
Dimitris Vartziotis and Joachim Wipper.
\newblock {The Geometric Element Transformation Method for Mixed Mesh
  Smoothing}.
\newblock {\em Eng. Comput.}, 25(3):287--301, 2009.

\bibitem{VartziotisWipper2010}
Dimitris Vartziotis and Joachim Wipper.
\newblock Characteristic parameter sets and limits of circulant {H}ermitian
  polygon transformations.
\newblock {\em Linear Algebra Appl.}, 433(5):945--955, 2010.

\bibitem{VartziotisWipperPapadrakakis2013}
Dimitris Vartziotis, Joachim Wipper, and Manolis Papadrakakis.
\newblock Improving mesh quality and finite element solution accuracy by
  {GETM}e smoothing in solving the {P}oisson equation.
\newblock {\em Finite Elem. Anal. Des.}, 66:36--52, 2013.

\end{thebibliography}
\end{document}